\documentclass[11pt]{article}

\usepackage{amsmath,amssymb,amsthm}

\usepackage{graphicx}
\usepackage{subfigure}

\usepackage{caption}
\usepackage{cmap}
\usepackage{epstopdf}
\usepackage{longtable}
\usepackage{booktabs}
\usepackage{dsfont}
\usepackage[top=1in , bottom =1in , left=1in , right =1in]{geometry}

\makeatletter
\newcommand\figcaption{\def\@captype{figure}\caption}
\newcommand\tabcaption{\def\@captype{table}\caption}
\makeatother
\usepackage{picinpar}
\usepackage{array}
\usepackage{indentfirst}    
\usepackage{listings}
\usepackage{setspace}
\usepackage{xcolor}

\newtheorem{Theorem}{Theorem}[section]

\newtheorem{lem}[Theorem]{Lemma}

\theoremstyle{definition}

\newtheorem{defi}[Theorem]{Definition}
\newtheorem*{rmk}{Remark}

\title{\bf On the~$m$-point convexity}

\author{Wenzhi Liu$^a$\ \ \  Wei Wang$^b$\ \ \  Liping Yuan$^{a, b, c, d}$\thanks{Corresponding author. Email: wenzhiliu0601@163.com (WL); 15733133775@163.com (WW); lpyuan@hebtu.edu.cn (LY); tuzamfirescu@gmail.com (TZ).} \ \ \  Tudor Zamfirescu$^{a, e, f}$}

\begin{document}\large
	\date{}
	\vskip -8cm
	\maketitle

{\small $a$. School of Mathematical Sciences,
Hebei Normal University,
050024 Shijiazhuang, P.R. China.

$b$. Hebei International Joint Research Center for Mathematics and Interdisciplinary Science,
050024 Shijiazhuang, P.R. China.

$c$. Hebei Key Laboratory of Computational Mathematics and Applications,
050024 Shijiazhuang,  P. R. China.

$d$. Hebei Research Center of the Basic Dicipline Pure Mathematics, 050024 Shijiazhuang,  P. R. China.

$e$. Fakult\"{a}t f\"{u}r Mathematik, TU Dortmund,
44221 Dortmund, Germany.

$f$.  Roumanian
Academy, 010071 Bucharest,
Roumania.}

	\begin{abstract}
		
Let $S\subset \mathbb{R}^d$ $(d\geq 2)$. A set $S$ is said to be $m$-point convex, if for every $m$ distinct points in $S$, at least one of the line-segments determined by them lies in $S$. We also say that $S$ has property $P_m$. Let~${x,y,z}\in \mathbb{R}^{d}$.~If $\mathrm{conv}\{x,y,z\}$ is a right triangle,~then~$\{x,y,z\}$~is called a {\it right triple}. A set~$S$~is said to have the {\it right-$3$-point property},~if, for every right triple of~$S$,~at least one of the line-segments determined by them belongs to~$S$. In particular, it has the {\it double right-$3$-point property},~if, for every right triple in~$S$,~at least two of the line-segments determined by them belong to~$S$. In this paper, we further investigate $m$-point convex sets and establish the relationship between the sets with the double right-$3$-point property and convex sets in $\mathbb{R}^d$.
		
		\bigskip
		
		
		\bigskip
		
	\end{abstract}
	
	\section{Introduction}
	In 1957, Valentine \cite{ref2} introduced the notion of $3$-point convexity and studied the decomposition of $3$-point convex sets in the plane. In 1967, Guay \cite{GM} generalized the concept of $3$-point convexity to $m$-point convexity (calling it $m$-convexity). In 1970, Guay and Kay established a connection between $m$-point convex sets and $L_n$ sets. The decomposition of $m$-point convex sets plays a crucial role in $m$-point convexity. Breen studied it in the plane \cite{BM1} and in $\mathbb{R}^d$ \cite{BM3}. In 1990, Perles and Shelah \cite{PS} proved that a closed $m$-point convex set in the plane can be covered by at most $(m-1)^6$ convex sets. In 1999, Matou\v sek and Valtr \cite{MV} improved this bound to $18(m-1)^3$. In 2013, Nitzan and Perles \cite{NP} proved that a $3$-point convex set in the plane can be decomposed into six convex sets.
	
	In 1962, Marr \cite{ref1} introduced the isosceles $3$-point property and studied its relationship with convex sets and $3$-point convex sets. In 1994, Latecki, Rosenfeld and Silverman \cite{LRS} proposed the concept of $CP_3$, the property that for every three collinear points in the set, at least one of the line-segments determined by them lies in the set, and studied the relationship between sets with this property and convex sets.
	
	In this paper, we further investigate $m$-point convex sets and $m$-starshaped sets. Furthermore, we introduce the right-$3$-point property and the double right-$3$-point property, and investigate the relationship between the sets with the double right-$3$-point property and convex sets in $\mathbb{R}^d$.
	
	\section{Definitions and notation}
	Let~$S$~be a set in~$\mathbb{R}^{d}$~and~$m\geq 2$.~The set~$S$~is said to be~{\it$m$-point convex},~if, for every~$m$~distinct points in~$S$,~at least one of the line-segments determined by them lies in~$S$~\cite{ref7}.~We also say that~$S$~has {\it property}~$P_m$. Hence,~$2$-point convexity means usual convexity.
	
	Let~${x,y,z}$~in~$\mathbb{R}^{d}$.~If $\mathrm{conv}\{x,y,z\}$ is a right triangle,~then~$\{x,y,z\}$~is called a {\it right triple}~\cite{Y}.
	
	A set~$S$~has the {\it right-$3$-point property},~if, for every right triple of~$S$,~at least one of the line-segments determined by them belongs to~$S$. In particular, it has the {\it double right-$3$-point property},~if, for every right triple in~$S$,~at least two of the line-segments determined by them belong to~$S$.
	
	Generalizing the usual notion of an extreme point from convexity, we say that a point~$x$~in~$S\subset\mathbb{R}^{d}$~is an {\it extreme point} of~$S$~if there exists no nondegenerate line-segment in~$S$~that contains~$x$~in its relative interior.
	
	A set~$S$~is {\it $m$-starshaped} relative to a point~$x\in S$,~if for any~$m-1$~points~$x_1$, $x_2,...,x_{m-1}\in S$,~at least one of the line-segments ~$xx_1$,$xx_2$,...,$xx_{m-1}$ lies in $S$.~The points which can play the role of~$x$~form the~$m$-{\it kernel} $\mathrm{ker}_{m}S$ of~$S$. Thus, 2-starshapedness means usual starshapedness, and $\mathrm{ker}_{2}S=\mathrm{ker}S$.
	
	A tiling graph is a tiling of the plane viewed as a graph. Given a tiling graph, we call a line~$L$~a {\it tiling line},~if~$L$~only consists of edges and vertices of that tiling graph.
	
	
	Let~$S$~be a set in~$\mathbb{R}^{d}$.~$S$~is said to be {\it locally convex} at a point~$q\in S$,~if there exists an open ball~$N$~with center at~$q$~such that~$S\cap N$~is convex.
	
	Put $a_{1}a_{2}\ldots a_{n}=\mathrm{conv}\{a_{1},a_{2},\ldots ,a_{n}\}$, $\overline{a_{1}a_{2}\ldots a_{n}}=\mathrm{aff}\{a_{1},a_{2},\ldots ,a_{n}\}$ and $\left [a_{1},a_{2},\ldots ,a_{n}\right ]=a_{1}a_{2}\cup a_{2}a_{3}\cup \ldots \cup a_{n-1}a_{n}$, for $a_{1},\ldots ,a_{n}\in \mathds{R}^{d}$.
	
	For distinct points~$x,y\in \mathbb{R}^d$,~let ${L}_{xy}$ denote the hyperplane through $(x+y)/ 2$ orthogonal to $\overline{xy}$, $H_{xy}$ the hyperplane through $x$ orthogonal to $\overline{xy}$,  $C_{xy}$ the  hypersphere of diameter $xy$, and $W_{xy}=(C_{xy}\cup H_{xy}\cup H_{yx})\setminus \{x,y\}$. Put~$(xy)=xy\backslash \{x,y\}$.
	
	For~$S\subset \mathbb{R}^{d}$,~we denote by~$\mathrm{cl}S,\mathrm{int}S,\mathrm{bd}S,\mathrm{conv}S,\mathrm{dim}S$~its closure,~relative interior,~relative boundary,~convex hull,~dimension,~respectively.
	
	\section{~$m$-point convexity}

%




%

	We start with two simple remarks.
	
		If~$S\subset\mathbb{R}^{d}$~is a convex set and~$p\in S$,~then~$S\backslash \{p\}$~is~$3$-point convex.

		In the tilings~$(3^6),(4^4)$~and~$(3,6,3,6)$, a subset~$S$~is~$3$-point convex~if and only if~$S$~is a connected subset of a tiling line or the union of two such subsets.

	\begin{Theorem}\label{6}
		If~$S\subset\mathbb{R}^{d}$~is a simple closed curve,~then~$S$~is~$(n+1)$-point convex if and only if~$S$~is a closed broken line with $m$ vertices,~where~$m\leq n$.
	\end{Theorem}
	
	\begin{proof}
		The ``if\," implication is obvious.
		
		Next we prove the ``only if\," implication.
		
		Assume~$S\subset\mathbb{R}^{d}$~is a simple closed curve.~Choose any~$n+1$~points in~$S$.~At least one of the line-segments determined by them lies in~$S$.~We denote it by~$\Sigma$.~Let~$\Sigma_1$~be the longest line-segment in~$S$~which includes~$\Sigma$.~Now,~choose any~$n+1$~points in~$S\backslash\mathrm{int}\Sigma_1$~.~We get another (maximal) line-segment~$\Sigma_2$,~included in~$S\backslash \mathrm{int}\Sigma_1$.~Then,~we choose~$n+1$~points in~$S\backslash(\mathrm{int}\Sigma_1\cup \mathrm{int}\Sigma_2)$,~and continue this way,~until obtaining~$m$~line-segments~$\Sigma_1,\Sigma_2,...,\Sigma_m$,~the union~$V_m$~of which equals ~$S$. ~Indeed, if~$V_m\ne S$~for all~$m\leq n$,~consider a point~$v$~in~$S\backslash V_n$.~The property~$P_{n+1}$~is not verified at~$v$~and the midpoints of~$\Sigma_1,\Sigma_2,...,\Sigma_n$.~Hence,~$V_m=S$~for some~$m\leq n$,~and~$S$~is the broken line $V_m$.
	\end{proof}
	
	\begin{Theorem}
		Let~$S\subset\mathbb{R}^{d}$~be a simple arc.~Then~$S$~is~$m$-point convex if and only if~$S$~is the union of~$m-1$~line-segments.
	\end{Theorem}

	\begin{proof}
		The proof is similar to the proof of Theorem \ref{6}.
	\end{proof}

	\begin{Theorem}
		Let~$S\subset\mathbb{R}^{2}$~be a topological disc.~If~$S$~has at most one point at which it is not locally convex,~then~$S$~is~$3$-point convex.
		
	\end{Theorem}
	
	\begin{proof}
		Let~$q\in S$~be a point at which~$S$~is not locally convex.~Choose~$p\in \mathrm{bd}S$~close to~$q$.~Take a point~$r$~close to~$q$~such that~$q\in \overset{\LARGE{\frown}}{pr}$.
		
		Let~$\tau_p$~be the tangent of~$\overset{\LARGE{\frown}}{qp}$~at~$q$~and~$\tau_r$~the tangent of~$\overset{\LARGE{\frown}}{qr}$~at~$q$.~Now take the bisector of the angle made by~$\tau_p,\tau_r$.

		The set~$S$~is divided by the bisector into two convex sets,~by Tietze's theorem~\cite{ref3}.~It follows that~$S$~is~$3$-point convex.
		
		If $S$ is everywhere locally convex, then it is convex and the conclusion is evident.
	\end{proof}
	
	\begin{Theorem}\label{1}
		Every $3$-point convex continuum in the plane is simply connected.
		
	\end{Theorem}

    \begin{proof}
    	We assume that $S$ is not simply connected. So, the complement of $S$ has an open, bounded component $C$. Choose a maximal disk $D\subset \mathrm{cl}C$.  If $\mathrm{card}(D\cap S)=2$, then there exist two diametrally opposite points $x,y$ on the boundary of $D$, which belong to $S$. Consider $z\in {L}_{xy}\cap S$. Obviously, $xy,xz,yz\not \subset S$. If $u,v,w\in D\cap \mathrm{bd}S$, then $uv,uw,vw\not \subset S$.
    \end{proof}
	
	Let ${B}_{S}$ be the set of all points of $S\subset \mathbb{R}^2$ at which $S$ is not locally convex.
	
	\begin{lem}\label{3}
		If $S$ is a $3$-point convex continuum in the plane, then ${B}_{S}\subset \mathrm{ker}S$.
	\end{lem}
	
	\begin{proof}
		For every $x\in {B}_{S}$ and any $n\in N$, there exist points ${x}_{n}, {x}_{n}^{\prime}\in B(x,\frac{1}{n})$ such that ${x}_{n}{x}_{n}^{\prime}\not \subset S$, since $S$ is not locally convex at $x$. By the $3$-point convexity of $S$, ${x}_{n}^{\prime\prime}y$ must lie entirely in $S$ for any $y\in S$, where ${x}_{n}^{\prime\prime}$ is ${x}_{n}$ or ${x}_{n}^{\prime}$. Clearly, ${x}_{n}^{\prime\prime}y\to xy$ if $n\to \infty$. Moreover, $xy\subset S$, since $S$ is a continuum. Thus, $x\in \mathrm{ker}S$.
	\end{proof}
	
	 We say that a set of points $S$ in the Euclidean space $\mathbb{R}^{d}$ is {\it in general position} if any $d+1$ points of $S$ are affinely independent, $S$ is {\it in convex position} if no point of $S$ is a convex combination of the other points.
	
	\begin{lem}\label{4}
		If $S$ is a $3$-point convex continuum in the plane, then ${B}_{S}$ is in general position.
	\end{lem}
	
	\begin{proof}
		Assume there exist $x,y,z\in {B}_{S}$, such that $y\in xz$. If ${y}^{+}\in S$ is close to $y$ and lies on one side of $\overline{xz}$, then $x{y}^{+}\cup {y}^{+}z\cup xz\subset S$ because $x,z\in \mathrm{ker}S$, by Lemma \ref{3}. By Theorem \ref{1}, $x{y}^{+}z\subset S$.
		
		If there is no point ${y}^{-}\in S$ close to $y$ on the other side of $\overline{xz}$, then $S$ is locally convex at $y$, absurd.
		
		If there is such a point ${y}^{-}$, then $x{y}^{-}z\subset S$ and therefore $y\notin \mathrm{bd}S$, absurd too.
		
		If neither ${y}^{+}$, nor ${y}^{-}$ exists. then $y\notin {B}_{S}$, a contradiction.
	\end{proof}

%
%

	\begin{Theorem}
	If $S$ is a $3$-point convex continuum in the plane, then ${B}_{S}$ is a totally disconnected set in convex position.
		
	\end{Theorem}
	
	\begin{proof}
	By Theorem \ref{1}, $S$ is simply connected. By Lemma \ref{3}, ${B}_{S}\subset \mathrm {ker}{S}\subset S$. Thus, $\mathrm{conv}{B}_{S}\subset \mathrm{ker}S$. Suppose there exists a non-degenerate arc $A\subset {B}_{S}$. By Lemma \ref{4}, $A$ is not a line-segment. Take arbitrarily $x\notin \mathrm{conv}A$, close to $A$. Consider ${A}^{\prime}=\mathrm{conv}(\{x\}\cup A)$. If $x\in S$, then ${A}^{\prime}\subset S$, because ${A}^{\prime}\subset \mathrm{ker}S$. Hence, many points in $A$ are not boundary points of $S$. Absurd.	
	\end{proof}

	\begin{lem}\label{2}
		For any line-segment~$L\subset \mathbb{R}^{d}$,~the set~$\mathbb{R}^{d}\backslash L$~can be partitioned in two convex sets.
		
	\end{lem}
	\begin{proof}
		Let~$d=1$.~Obviously,~$\mathbb{R}^{1}\backslash L=A\cup B$,~where~$A,B$~are two open half-lines.
		
		
		Suppose the theorem holds when~$d=n-1$.~Now,~we assume~$d=n$.~Let~$H\supset L$~be a hyperplane of~$\mathbb{R}^{d}$,~and~$H^+,H^-$~the open half-spaces determined by~$H$.~According to our assumption,~$H\backslash L=S_1\cup S_2$,~where~$S_1,S_2$~are convex.~Then~$\mathbb{R}^{n}\backslash L$~is partitioned in two convex sets,~$S_1\cup H^+$~and~$S_2\cup H^-$.
	\end{proof}

	\begin{Theorem}\label{5}
		Assume~$A\cup B=D\subset\mathbb{R}^{d}$.~If~$A$~is~$m$-point convex and~$B$~is~$n$-point convex,~then~$D$~is~$(m+n-1)$-point convex.
	\end{Theorem}

	\begin{proof}
		Let~$x_1,x_2,...,x_{m+n-1}$~be~$m+n-1$~points in~$D$.
		
		Assume at least~$m$~of these points are in~$A$.~Since~$A$~is~$m$-point convex,~there is a line-segment~$x_ix_j\subset A \subset D$.~In case~$A$~contains at most~$m-1$~points,~$B$~contains at least~$n$~points.~Since~$B$~is~$n$-point convex,~we have a line-segment~$x_ix_j\subset B \subset D$.
		
		Hence~$D$~is~$(m+n-1)$-point convex.
	\end{proof}
	
	\begin{Theorem}
		Let~$S\subset\mathbb{R}^{d}$~be an~$m$-point convex set,~and~$p,q$~any two points in~$S$.~Then
		~$S\backslash pq$~is~$(2m-1)$-point convex.
	\end{Theorem}

	\begin{proof}
		By Lemma \ref{2},~$\mathbb{R}^{d}\backslash pq=A'\cup B'$,~where~$A',B'$~are two convex sets.\\~Then the sets~$A=S\cap A',B=S\cap B'$~are~$m$-point convex.~So we have~$S\backslash pq=A\cup B$.~From Theorem \ref{5},~it follows that~$A\cup B$~is~$(2m-1)$-point convex.
	\end{proof}

	\begin{Theorem}
		
		If~$S\subset\mathbb{R}^{d}$~is a convex set,~$d\geq 2$,~and~$A\subset S$~is a polytope with~$m$~facets,~then~$S\backslash A$~is~$(m+1)$-point convex.
	\end{Theorem}
	
	\begin{proof}
		Let~$A_1,A_2,...,A_m$~be the~$m$~facets of~$A$.
		
		Consider the supporting hyperplane~$H_i\supset A_i$.~Suppose~$H_i^+$~is the open half-space determined by~$H_i$~and disjoint from~$A$.~Put~$S_i=S\cap H_i^+$.

		Now,~$S\backslash A=\mathop{\cup}\limits _{i=1}^{m}S_i$.~Among~$m+1$~points from~$S\backslash A$,~there must exist two points in the same~$S_i$.~The line-segment joining them lies in~$S_i\subset S\backslash A$.
		
		Hence,~$S\backslash A$~is~$(m+1)$-point convex.
	\end{proof}

	\begin{Theorem}
		The closure of an~$m$-point convex set in~$\mathbb{R}^{d}$~is~$m$-point convex.
	\end{Theorem}
	
	\begin{proof}
		We prove the theorem for~$m=3$.~Adapting the proof to  arbitrary~$m$~is straightforward.
		
		Let~$S\subset \mathbb{R}^{d}$~be a~$3$-point convex set.~Choose~$x_0,y_0,z_0\in S$,~$x_1,y_1,z_1\in S$,...,~$x_n,y_n,z_n\in S$;~we have at least one line-segment contained in~$S$~for every triple~$x_i,y_i,z_i$.
		
		Take~$x_n\rightarrow x,y_n\rightarrow y,z_n\rightarrow z$.~For every~$n$,~$x_ny_n\subset S$,~or~$x_nz_n\subset S$,~or~$y_nz_n\\\subset S$.~Then~$x_ny_n\subset S$~for infinitely many indices~$n$,~or~$y_nz_n\subset S$~for infinitely many indices~$n$,~or~$x_nz_n\subset S$~for infinitely many indices~$n$.~Without loss of generality,~suppose that~$x_ny_n\subset S$~for infinitely many indices~$n$.~Since~$x_n\rightarrow x$~and~$y_n\rightarrow y$,~we have~$x_ny_n\rightarrow xy$.~Finally,~$x_ny_n\subset S$~implies~$xy\subset \mathrm{cl}S$.
		
		So,~$\mathrm{cl}S$~is a~$3$-point convex set.
	\end{proof}

	Notice that	any~$n$-starshaped set is~$m$-starshaped,~for every integer~$m>n$.

	\begin{Theorem}
		If~$S\subset \mathbb{R}^{2}$~is an open connected~$3$-starshaped set,~then~$S$~is a simply connected set or a simply connected set minus a single point set.
		
	\end{Theorem}

	\begin{proof}
		Let~$S$~be an open connected~$3$-starshaped set.
		
		Assume~$S$~is not simply connected. So, the complement of $S$
		has a closed, bounded component $C_1$. If $\mathrm{card}(C_1)\geq 2$,~then there are at least two points~$s,t\in C_1$. For any point $x\in S$, there are two points $y,z\in S$ such that~$s\in xy$ and $t\in xz$. Hence, $xy,xz\not\subset S$. Therefore, there is no point in the $3$-kernel of~$S$, absurd. If $\mathrm{card}(C_1)=1$ and there exists another closed, bounded component $C_2$ of the complement of $S$, then let $C_1=\{s\}$ and $C_2=\{t\}$. The preceding argument applies again.
		
		Hence,~$S$~is a simply connected set or a simply connected set minus the single point set $C_1$.
	\end{proof}

	\begin{Theorem}
		If~$S\subset \mathbb{R}^{d}$~is~$3$-starshaped,~then~$S$~has at most two components. If ~$S$~has two components,~one of the components is a single point set and the other a star-shaped set.
	\end{Theorem}

	\begin{proof}

		Clearly,~$S$~can not have at least~$3$~components.~So,~$S$~has at most two components,~$S_1,S_2$.

		Assume that neither~$S_1$~nor~$S_2$~is a single point set.~For any points~$x_1\in S_1$,~$y_1,z_1\in S_2$,~we have~$x_1y_1,x_1z_1\not\subset S$.~Similarly,~for any points~$x_2\in S_2$,~$y_2,z_2\in S_1$,~$x_2y_2,x_2z_2\not\subset S$.~So~$S$~has an empty kernel.
		
		Assume~$S_1=\{m\}$,~$S_2$~are the components of~$S$.~Of course,~the~$3$-kernel~$K$~\\of~$S$~lies in~$S_2$.~If~$kx\not\subset S_2$~for some~$k\in K$~and~$x\in S_2$,~the points~$m$~and~$x$~can not be seen from~$k$,~which contradicts the~$3$-starshapedness of~$S$.
	\end{proof}

\section{Double right-$3$-point property}

	\begin{Theorem}\label{9}
	Any continuum with non-empty interior and with the double right-$3$-point property is convex.	
\end{Theorem}
	
\begin{proof}
	Let $S$ be a continuum with non-empty interior and with the double right-$3$-point property. Firstly, we show that $S=\mathrm{cl}~\mathrm{int}S$. Take $a\in S$ and $b\in \mathrm{int}S$. There exists  $\delta > 0$ such that $B(b,\delta)\subset S$. Choose two points $c,{c}^{\prime}\in B(b,\delta)$ such that $\{a,c,{c}^{\prime}\}$ be a right triple. Since $S$ has the double right-$3$-point property, at least one of the line-segments $ac$, $a{c}^{\prime}$ lies in $S$. As $\delta$ could be taken as small as we wish, and $S$ is compact, this implies that $ab\subset S$. Consequently, $\mathrm{conv}(\{a\}\cup\mathrm{int}~B(b,\delta))\subset S$, whence $a\in \mathrm{cl}~\mathrm{int}S$.
	
	Now, we show that $S$ is convex. Let $x,y\in S$. Since $x\in \mathrm{cl}~\mathrm{int}S$, we find $x'$ close to $x$ and $\epsilon >0$, such that $B(x', \epsilon)\subset S$. As above, $\mathrm{conv } (\{y\}\cup \mathrm{int}  B(x', \epsilon))\subset S$. As $\|x-x'\|$ can be chosen arbitrarily small and $S$ is compact, this yields $xy\subset S$.
\end{proof}	
\begin{Theorem}\label{8}
	If $S\subset \mathbb{R}^{d}$ has non-empty interior and enjoys the double right-$3$-point property, then $S$ is connected and satisfies $S\subset \mathrm{cl}~\mathrm{int}S$.
\end{Theorem}

\begin{proof}
	If $S$ is disconnected, then $S$ must have a component $C_1$ with non-empty interior. Let $a \in \mathrm{int} C_1$ and $b$ belong to another component $C_2$. There exists $\epsilon > 0$ such that $B(a, \epsilon) \subset \mathrm{int} C_1$. Take $c \in H_{ab} \cap B(a, \epsilon)$. The set $\{a, b, c\}$ forms a right triple. However, $ab,bc\not \subset S$ because $\{a,c\}$ and $\{b\}$ are in different components. This contradicts the double right-$3$-point property of $S$.
	
	Let $x\in S$, $y\in \mathrm{int}S$ and $B(y,\delta )\subset S$. If the ``bounded cone" $\mathrm{conv}(\{x\}\cup\mathrm{int}~B(y,\delta))$ is included in $S$, then $x\in \mathrm{cl}~\mathrm{int}S$. If not, choose $z\in \mathrm{int}~B(y,\delta)$ such that $xz\not\subset S$. Like in the preceding proof, we see that there exists a ``bounded cone" with apex $x$, with $xz$ in its boundary, and with its interior included in $S$. Hence, $x\in \mathrm{cl}~\mathrm{int}S$.
\end{proof}

	\begin{Theorem}\label{14}
		Let~$S\subset\mathbb{R}^{d}$~be an open set.~Then~$S$~has the double right-$3$-point property,~if and only if~$S$~is convex or a convex set minus a single point set.	
	\end{Theorem}
	
	\begin{proof}
		
		The ``if\," implication is obvious.
		
		For the ``only if\," implication,~assume~$S$~is not convex and let~$x,y\in S$~satisfy~$xy\not \subset S$.
		
		Since~$S$~is an open set,~there is a neighborhood~$B\subset S$~of~$x$.~Take two points~$w,z\in H_{xy}\cap B$~such that~$ x\in wz$.~Obviously,~$w,x,y$~and~$z,x,y$~form two right triples,~respectively.~Since~$xy\not \subset S$,~we get~$wx\cup wy\subset S$~and~$zx\cup zy\subset S$.~Moreover,~$wzy\backslash xy \subset S$,~because~$uy\subset S$~for all~$u\in wz\backslash \{x\}$.
		
		Assume that there are two points~$s,t\in xy\backslash S$.~Choose~$a\in zxy$, $b,c\in wzy$~such that~$s\in ab,t\in ac$~and~$\angle bac=\frac {\pi}{2}$. The double right-$3$-point property of~$S$ is contradicted.~Hence,~$xy$~contains just one point of the complement of~$S$.
		
		Let~$S'=S\cup \{s\}$,~where~$s$~is the single point of~$xy\backslash S$~found above.~Note that~$S'$~is open.~Assume~$S'$~is not convex.~Then there are two points~$x',y'\in S'$~such that~$x'y'\not \subset S'$.~There exists a point~$s'\in x'y'\backslash S'$.~Take two neighborhoods~$N\subset S'$~of~$s$~and~$N' \subset S'$~of~$s'$.~Then there must exist two points~$t \in N \backslash\{s\},t' \in N'\backslash\{s'\}$~such that~$s,s'\in tt'$,~which contradicts the fact obtained above that~$tt'$~contains just one point of the complement of~$S$.
		
		Hence~$S'$~is convex,~whence~$S$~is a convex set or a convex set minus a single point set.
	\end{proof}
	
	A polygonal path $P=\left [a_{1},a_{2},\ldots ,a_{n}\right ]$ in $\mathbb{R}^d$ is a
	{\it staircase}, if every edge of $P$ is parallel to one of the coordinate axes, and if when drawing the path, all its parallel edges point in the same
	direction \cite{BM}.
	
	We leave the easy proofs of the following two results to the reader.
	\begin{Theorem}\label{12}
		Any simple arc with the double right-$3$-point property in the plane either does not contain right triples or, after an appropriate rotation, contains a staircase.
	\end{Theorem}



 \begin{Theorem}
 	Let $S\subset \mathbb{R}^{2}$ be a simple arc with the double right-$3$-point property. If $\{a, b, c\}$ and $\{b, c, d\}$ are right triples in $S$, then, after an appropriate rotation, $[a, b, c, d]$ is a staircase included in $S$.
 \end{Theorem}

%
%
	
\begin{Theorem}\label{10}
	If a bounded set $S \subset \mathbb{R}^d$ with non-empty interior has the double right-$3$-point property, then $\mathrm{cl}S$ also has the double right-$3$-point property.
\end{Theorem}
	
	\begin{proof}
		By Theorem \ref{8}, $S$ is connected and satisfies $\mathrm{cl}S = \mathrm{cl}~\mathrm{int}S$, whence $\mathrm{int}~\mathrm{cl}S\supset \mathrm{int}S\ne \emptyset$. By Theorem \ref{9}, if $\mathrm{cl}S$ does not have the double right-$3$-point property, then it is not convex. By Tietze's theorem, there exists a point $a \in \mathrm{cl}S$ such that $\mathrm{cl}S$ is not locally convex at $a$. For any $\epsilon > 0$, there exist points $x, y \in B(a, \epsilon) \cap \mathrm{int}S$ such that $xy \not \subset S$. Consider a point $z$ close to $y$, such that $\overline{xy} \perp \overline{yz}$ and $xz \not\subset S$. $\{x, y, z\}$ forms a right triple. However, $xy, xz \not\subset S$. This is a contradiction.
	\end{proof}
	
\begin{Theorem}\label{11}
	If a bounded set $S \subset \mathbb{R}^d$ with non-empty interior has the double right-$3$-point property, then $\mathrm{int}S$ also has the double right-$3$-point property.
\end{Theorem}
	\begin{proof}
		We start with the following.

\underline{Remark}. If $u,v\notin S$, we cannot have both $B(u,\epsilon )\setminus \{u\}\subset S$ and $B(v,\epsilon )\setminus \{v\}\subset S$.

Indeed, we take points $p,p'\in S$ close to $u$ and $q\in S$ close to $v$ such that $u\in pp'$, $v\in pq$ and $\overline{pp'} \perp \overline{pq}$. Obviously, $pp'\not \subset S$ and $pq\not \subset S$. This contradicts the double right-$3$-point property of $S$.
		
		By Theorem \ref{14}, we just need to prove that $\mathrm{int}S$ is convex or a convex set minus a single point set.
		
		Assume that $\mathrm{int}S$ is not convex. Then there exist points $x,y\in \mathrm{int}S$ such that $xy\not\subset \mathrm{int}S$. Let $z\in xy \setminus \mathrm{int}S$. We claim that $z\notin S$.
		
		Let us prove it. If $z\in S$, then $z\in \mathrm{bd}S$. Since $x,y\in \mathrm{int}S$, there exists $\epsilon > 0$ such that $B(x,\epsilon )\subset S$ and $B(y,\epsilon )\subset S$. There exists a point ${z}^{\prime}\in B(z,\epsilon )$ such that ${z}^{\prime}\notin S$. We can find points ${x}^{\prime}\in B(x,\epsilon )$ and ${y}^{\prime}\in B(y,\epsilon )$ such that ${z}^{\prime}\in {x}^{\prime}{y}^{\prime}$.
		
			\begin{figure}[htbp]
				\centering
				\includegraphics[width=9cm]{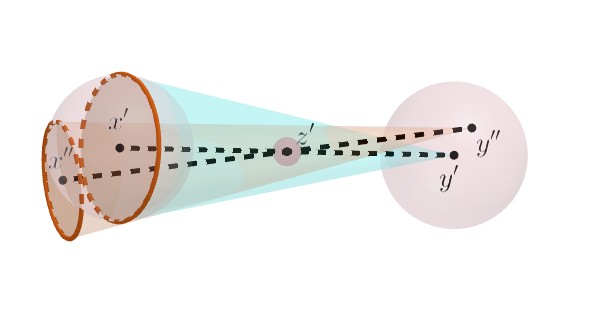}
				\caption{Illustration for Theorem \ref{11}.}
				\label{Theorem 3.8}
			\end{figure}
		
		
		Consider ${\epsilon}^{\prime}> 0$ such that $B({x}^{\prime},{\epsilon}^{\prime})\subset S$. By the double right-$3$-point property of $S$, it follows that $\mathrm{conv}(({H}_{{x}^{\prime}{y}^{\prime}}\cap B({x}^{\prime},{\epsilon}^{\prime}))\cup \{{y}^{\prime}\})\setminus {x}^{\prime}{y}^{\prime}\subset S$. We find points ${x}''\in B({x}^{\prime},\epsilon )$ and ${y}''\in B({y}^{\prime},\epsilon )$ such that $\{z^{\prime}\}={x}^{\prime}{y}^{\prime}\cap {x}''{y}''$. Analogously, $\mathrm{conv}(({H}_{{x}''{y}''}\cap B({x}'',\nu ))\cup \{{y}''\})\setminus {x}''{y}''\subset S$, for some $\nu > 0$. It follows that there exists $\delta > 0$ such that $B({z}^{\prime},\delta )\setminus \{{z}^{\prime}\}\subset S$.
		
		Let ${\epsilon }_{1}<\| z-{z}^{\prime} \|$. Analogously, we can find ${z}_{1}\in B(z,\epsilon _{1})\setminus S$ and ${\delta }> 0$, such that $B({z}_{1},{\delta })\setminus \{{z}_{1}\}\subset S$. This contradicts our Remark. This proves the claim.
		
         Now, we prove that the point $z$ is the only one in $xy\setminus (\mathrm{int}S)$. Indeed, if $z_2\in xy$ is another point with $z_2\notin \mathrm{int}S$, then $z_2\notin S$ by our claim. By the double right-$3$-point property of $S$, we conclude that $\mathrm{conv}(({H}_{{x}{y}}\cap B({x},\epsilon))\cup \{y\})\setminus xy\subset S$. Applying the preceding argument again, there exists ${\delta}_1>0$ such that $B(z,{\delta}_1)\setminus \{z\}\subset S$ and $B(z_2,{\delta}_1)\setminus \{z_2\}\subset S$. By the Remark, this is impossible.

         Next, we will show that, for any two points $x_1, y_1\in \mathrm{int}S$, if $z\notin x_1y_1$, then $x_1y_1\subset \mathrm{int}S$. Suppose that this is not the case. Then, by an analogous argument, we can find another point $t\in x_1y_1$ such that a whole neighborhood of $t$ minus $\{t\}$ is included in $S$. Since $B(z,{\delta})\setminus \{z\}\subset S$, a contradiction is provided by our Remark.
		
		 Hence, $\mathrm{int}S\cup \{z\}$ is convex. Consequently, $\mathrm{int}S$ is a convex set minus the single point set $\{z\}$.
	\end{proof}
%
%

In fact, Theorems \ref{10} and \ref{11} show that even a bounded set that is neither open nor closed, but has non-empty interior and the double right-$3$-point property, is still ``well-behaved''. Its closure is convex, and its interior is convex or a convex set minus a single point set.

\begin{Theorem}
		Any compact set in $\mathbb{R}^d$ with the double right-$3$-point property, which includes a right triple $\{a,b,c\}$ and meets $\mathrm{int}~abc$, is convex.
\end{Theorem}
	
\begin{proof}
		\begin{figure}[htbp]
		\centering
		\includegraphics[width=7cm]{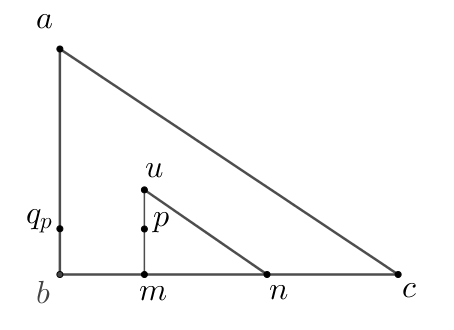}
		\caption{$d=2$ and $\mathrm{int}S \ne \emptyset$.}
		\label{Theorem 3.10}
		\end{figure}
First, suppose $d=2$ and let $S$ be a set satisfying the conditions of the statement. Take $u \in S$ such that $u \in \mathrm{int}~abc$. Take $m \in (bc)$ such that $\overline{um}\perp \overline{bc}$. We claim that $um \subset S$. Indeed, if $um \not \subset S$, then, for any $n \in (mc)$, $un \subset S$ by the double right-$3$-point property of $S$, and so $um \subset S$ by the compactness of $S$, a contradiction.
		
For every $p \in um$, take $q_p \in ab$ such that $\overline{pq_p }\perp \overline{ab}$. Analogously, we have $pq_p \subset S$. Therefore, $ubm \subset S$, and furthermore, $\mathrm{int}S \ne \emptyset$. By Theorem \ref{8}, $S$ is a continuum. By Theorem \ref{9}, $S$ is convex.
		
Now, let $S\subset \mathbb{R}^d$. If $S\subset \overline{abc}$, we are ready. If not, still $S\supset abc$. Take $v\in S\setminus \overline{abc}$ and $u\in \mathrm{int}~abc$. We show that $vu\subset S$. Indeed, consider $w\in \mathrm{int}~abc$ close to $u$ such that $\overline{wu} \perp \overline{vu}$. By the double right-$3$-point property of $S$, $uv$ or $vw$ lies in $S$. This together with the compactness of $S$ imply $uv\subset S$. Hence, $\mathrm{conv}(\{v\}\cup \mathrm{int}~abc)\subset S$. By Theorem \ref{8}, $S$ is a continuum, and by Theorem \ref{9}, it is convex.
	\end{proof}

\section{Right-$3$-point property}

\begin{lem}\label{13}
		Every simple closed curve in $\mathbb{R}^d$ contains a right triple.
	\end{lem}

	\begin{proof}
	Let $J$ be a simple closed curve in $\mathbb{R}^d$ and $x,y\in J$. If $W_{xy}\cap J\ne \emptyset $, then we are done. Suppose now that $W_{xy}\cap J= \emptyset $.
	
	Case 1. If $(\mathrm{int}~\mathrm{conv}{C}_{xy})\cap J\ne \emptyset$, then for any point $z\in (\mathrm{int}~\mathrm{conv}{C}_{xy})\cap J$, we have $H_{zx}\cap J\setminus \{z\}\ne \emptyset$, since $x$ and $y$ are on different sides of $H_{zx}$. Thus, for $a\in H_{zx}\cap J\setminus \{z\}$, $\{a,z,x\}$ forms a right triple.

	Case 2. If $(\mathrm{int}~\mathrm{conv}{C}_{xy})\cap J=\emptyset$, then we consider in $J$ a sequence of points ${\{{b}_{n}\}}_{n=1}^{\infty}$ converging to $x$.

This sequence has a subsequence ${\{{b}_{{n}_{m}}\}}_{m=1}^{\infty}$ such that $\frac{{b}_{{n}_{m}}-x}{\left \| {b}_{{n}_{m}-x}\right \|}$ converges, say, to $\tau $. So, w.l.o.g. we suppose that that subsequence is ${\{{b}_{n}\}}_{n=1}^{\infty}$ itself. Thus, for some index $s$ and any $n\geq s$, $\angle {b}_{n}x{b}_{s}< \frac{\pi }{4}$. This yields ${b}_{n}\in{C}_{x{b}_{s}}$, and we find ourselves in Case 1 with respect to $x, {b}_{s}$.
\end{proof}

	\begin{Theorem}\label{7}
	Simple closed curves in $\mathbb{R}^d$ don't have the right-$3$-point property.
\end{Theorem}

\begin{proof}
	\begin{figure}[htbp]
			\centering
			\includegraphics[width=6cm]{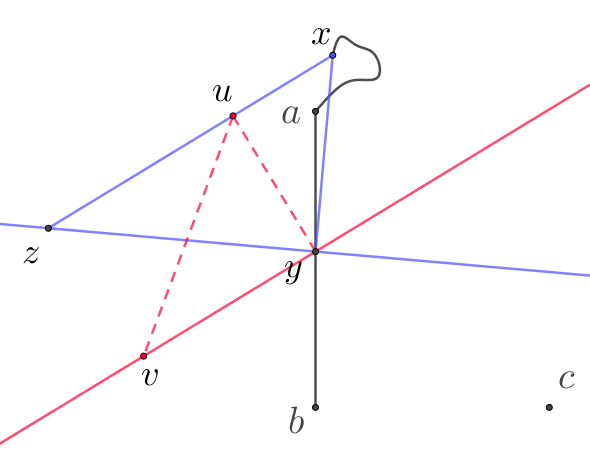}
			\caption{}
			\label{f}
		\end{figure}
	Assume that the simple closed curve $S$ has the right-$3$-point property. Let $\{a,b,c\}$ be a right triple of $S$; its existence is guaranteed by Lemma \ref{13}. At least one of the line-segments $ab,ac,bc$, say $ab$, is included in $S$, see Figure \ref{f}.  Without loss of generality, we suppose that there is no point $e\ne a$, such that $ab\subset be\subset S$. Take the point $x\in S\setminus ab$ close to $a$ and $y=\frac{1}{2}a+\frac{1}{2}b$. Since ${H}_{yx}$ separates $x$ from $b$, there exists a point $z\in {H}_{yx}\cap S\setminus \{y\}$. We have $yx, yz \not \subset S$, since $S$ is a simple closed curve. Thus, $xz\subset S$, by the right-$3$-point property. Let $u$ be the orthogonal projection of $y$ on $\overline{xz}$. Similarly, since ${H}_{yu}$ separates $u$ from $b$, we can find $v\in {H}_{yu}\cap S\setminus \{y\}$. Now, $uy,uv,vy\not\subset S$, and a contradiction is obtained.
\end{proof}

\begin{figure}[htbp]
			\centering
			\includegraphics[width=5cm]{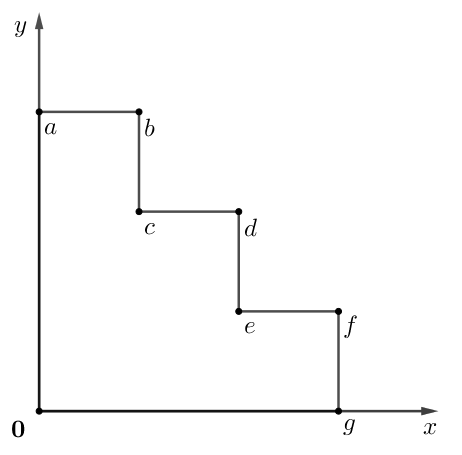}
			\caption{The right-$3$-point property does not imply $P_3$.}
			\label{e}
		\end{figure}

Obviously, if a set has the property $P_3$, then it has the right-$3$-point property. On the other hand, the right-$3$-point property does not imply $P_3$ even for a continuum with nonempty interior of the plane. The example in Figure \ref{e} illustrates this.

\bigskip
\bigskip

{\bf Acknowledgements.} This work was supported by the NSF of China (122711392); the Foreign Experts Program of the People's Republic of China; the Program for Foreign Experts of Hebei Province.

The first author (WL) was supported by the China Scholarship Council (grant No. 202508810001) and the 2026 Provincial Graduate Student Innovation Ability Training and Funding Project (grant No. CXZZBS2026076). The fourth author (TZ) has also been partially supported by the CNRS-IMAR cooperation in the frame of the IRN ECO-Math.

\end{document}